\documentclass[12pt]{article}

\usepackage{amsmath}
\usepackage{amsfonts}
\usepackage{amssymb}
\usepackage[colorlinks=true,linkcolor=black,citecolor=blue,urlcolor=blue,]{hyperref}
\newtheorem{theorem}{Theorem}[section]
\newtheorem{lemma}{Lemma}[section]

\newenvironment{proof}{\noindent{\bf Proof:}}{\hfill\fbox{}\vspace*{1mm}}

\begin{document}
	
\title{Toeplitz Operators on Weighted Bergman Spaces over Tubular Domains}

\author{ Lvchang Li, Jiaqing Ding, Haichou Li}
\author{Lvchang Li \thanks
	{College of Mathematics and Informatics,
		South China Agricultural University,
		Guangzhou,
		510640,
		China
		 Email: 20222115006@stu.scau.edu.cn.},\
	Jiaqing Ding\thanks{College of Mathematics and Informatics,
		South China Agricultural University,
		Guangzhou,
		510640,
		China
	Email: 20222115001@stu.scau.edu.cn.},\
	Haichou Li \thanks{ Corresponding author,
		College of Mathematics and Informatics,
		South China Agricultural University,
		Guangzhou,
		510640,
		China
		 Email: hcl2016@scau.edu.cn.  Li is supported by NSF of China (Grant No. 12326407 and 12071155). }
}

\date{}
\maketitle
\begin{center}
	\begin{minipage}{120mm}
		\begin{center}{\bf Abstract}\end{center}
		{In this paper, we mainly study the necessary and sufficient conditions for the boundedness and compactness of Toeplitz operators on weighted Bergman spaces over a tubular domains by using the Carlson measures on tubular domains. We also give some related results about Carlson measures.	}	
			
	\end{minipage}
\end{center}

\maketitle

\section{Introduction}
  \ \ \ \
  Bergman's book \cite{ber} systematically discusses a Hilbert space of square-integrable analytic functions on a 
  domain for the first time, now known as the Bergman space defined on a domain. The Bergman space is a closed subspace of the familiar $L^p$ space. When $p=2$, the Bergman space is a Hilbert space. A useful tool for studying the Bergman space is the reproducing kernel, which plays a very important role. For related theories on the reproducing kernel, please refer to relevant literatures\cite{aronszajn1950theory,ber,halmos2012hilbert,malyshev1995bergman}.

Another important tool in the study of operators on function spaces is the Carleson measure, which was initially introduced by Carleson \cite{carleson1958interpolation} to address the Corona problem. Nowadays, Carleson measures play a crucial role in studying the boundedness and compactness of operators, especially for Toeplitz operators, see \cite{zhu2005spaces,zhu2007operator}. Regarding the further applications of Carleson measures in operator theory on function spaces, refer to \cite{abate2011carleson,cima1982carleson,hastings1975carleson,luecking1983technique}.
The theory of Toeplitz operators on the unit disk and the unit ball in the Bergman spaces has been extensively studied by many authors, such as \cite{zhu2005spaces,zhu2007operator}. Subsequently, many authors have also extended the domains to bounded symmetric domains\cite{zhu1991hankel}, strongly pseudoconvex domains\cite{abate2019toeplitz,abate2012toeplitz,hu2016carleson}, pseudoconvex domains\cite{li2024carleson}, and so on.

 However, researches on the theory of the Bergman spaces on unbounded domains is scarce. In particular, when n=1, the Bergman spaces on the upper half-plane lacks many good properties of the Bergman spaces over the unit disk, such as the well-known constant functions and monomial functions not being in the Bergman spaces over the  upper half-plane. 

In the present paper, we are interested in the case of the higher dimensional unbounded domains, such as  the tubular domains. More specifically, this paper will mainly study the boundedness and compactness of Toeplitz operators on a certain class of tubular domains in $\mathbb{C}^n$ and their relationship with Carleson measures. These tubular domains may share some similarities with the well-known second kind of Siegel upper half-space, but the second kind of Siegel upper half-space are not tubes in $ \mathbb{C}^n$. Instead, the tubes are larger than the second kind of Siegel upper half-space, hence they have corresponding research value.

 Deng et al. \cite{deng2021reproducing} computed the reproducing kernel of the Bergman spaces on such tubes using Laplace transform methods, laying the groundwork for subsequent theoretical researches. Liu et al. \cite{jiaxin2023bergman} provided some basic properties of the Bergman spaces on these tubes.
Si et al. \cite{liu2020positive} studied the boundedness and compactness of Toeplitz operators on the Bergman space of the second kind of Siegel upper half-space and their relationship with Carleson measures. This paper will mainly study the theory of Toeplitz operators on the Bergman spaces over tubes with the help of Carleson measure, which is a  powerful tool and an interesting object to study.

The structure of the paper is as follows: The second section provides an overview of fundamental terminology. The third section presents essential lemmas and their proofs. In the fourth section, we obtain the characterization of Carleson measures on the tubular domains. Moving forward, in the fifth section, we can find a dense subspace of Bergman spaces over tubular domains, which is crucial for establishing the boundedness of Toeplitz operators. Finally, the last section comprehensively explores the characterizations of boundedness and compactness of Toeplitz operators on Bergman spaces over tubular domains, with detailed discussions on Theorem \ref{mianth1} and Theorem \ref{mianth2}.

\section{Preliminaries}

\ \ \ \
Let $\mathbb{C}^n$ be the $n$ dimensional complex Euclidean space. For any two points $z=\left( z_1,\cdots ,z_n \right) $ and $w=\left( w_1,\cdots ,w_n \right) $ in $\mathbb{C}^n,$ we write\[z\cdot \bar{w}:=z_1\bar{w}_1+\cdots +z_n\bar{w}_n,\] \[z^2=z\cdot z:=z_{1}^{2}+z_{2}^{2}+\cdots +z_{n}^{2}\] and \[\left| z \right|:=\sqrt{z\cdot \bar{z}}=\sqrt{\left| z_1 \right|^2+\cdots +\left| z_n \right|^2}.\]  

The set $\mathbb{B}_n=\left\{ z\in \mathbb{C}^n:\left| z \right|<1 \right\}$ will be called the unit ball of $\mathbb{C}^n$. 

The tubular domain $T_B$ of $\mathbb{C}^n$ with base $B$, is  defined as follows:

 $$ T_B=\left\{ z=x+iy \in \mathbb{C}^n | x\in \mathbb{R}^n, \  y\in B \subseteq  \mathbb{R}^n  \right\}, $$   
 where\[B=\left\{ \left( y',y_n \right)=( y_1,\cdots ,y_{n-1}, y_n) \in \mathbb{R}^n\left| y'^2:=y_{1}^{2}+\cdots +y_{n-1}^{2}<y_n \right. \right\}.\] 

We define the spaces $L_{\alpha}^{p}\left( T_B \right)$, which is composed of all Lebesgue measurable functions $f$ on $T_B$, and its norm \[\lVert f \rVert _{p,\alpha}=\left\{ \int_{T_B}{\left| f\left( z \right) \right|^p dV_\alpha\left( z \right)} \right\} ^{\frac{1}{p}}\] is finite, where $dV_\alpha(z)=( y_n-\left| y' \right|^2)^\alpha dV(z)$, $\alpha>-1$, $dV(z)$ denotes the Lebesgue measure on $\mathbb{C}^n$. 

The Bergman spaces $A_{\alpha}^{p}\left( T_B \right)$ on tube $T_B$  is a set composed of all holomorphic functions in $L_{\alpha}^{p}\left( T_B \right)$. 

Since the valuation functional is bounded, so the Bergman space $A_{\alpha}^{p}\left( T_B \right)$ is the closed subspace of $L_{\alpha}^{p}\left( T_B \right)$. At the same time, we know that when $1\le p<\infty $ the space $A_{\alpha}^{p}\left( T_B \right)$ is a Banach space with the norm $\lVert \cdot \rVert _{p,\alpha}.$ In particular, when $p=2,$ $A_{\alpha}^{2}\left( T_B \right)$ is a Hilbert space. 

 An very important orthogonal projection $P_{\alpha}$ from $L_{\alpha}^{2}\left( T_B \right)$ to $A_{\alpha}^{2}\left( T_B \right)$ is the following integral operator:\[P_{\alpha}f\left( z \right) =\int_{T_B}{K_\alpha\left( z,w \right) f\left( w \right) dV_\alpha\left( w \right)},\] with the Bergman kernel \[K_{\alpha}\left( z,w \right) =\frac{2^{n+1+2\alpha}\varGamma \left( n+1+\alpha \right)}{\varGamma \left( \alpha +1 \right) \pi ^n}\left( \left( z'-\overline{w'} \right) ^2-2i\left( z_n-\overline{w_n} \right) \right) ^{-n-\alpha -1}.\]  

For convenience, we introduce the following notation:
 \[\rho \left( z,w \right) =\frac{1}{4}\left( \left( z'-\overline{w'} \right) ^2-2i\left( z_n-\overline{w_n} \right) \right)\] and let $\rho \left( z \right) :=\rho \left( z,z \right) =y_n-y'^2.$ 

With the above notion $\rho \left( z,w \right)$, the weighted Bergman kernel of $T_B$ becomes
 \[K_{\alpha}\left( z,w \right) =\frac{\varGamma \left( n+\alpha +1 \right)}{2^{n+1}\pi ^n\varGamma \left( \alpha +1 \right) \rho \left( z,w \right) ^{n+\alpha +1}}.\]

Recall $\rho \left( z \right) =y_n-y'^2$ and let $\partial T_B:=\left\{ z\in \mathbb{C}^n\,\,: \rho \left( z \right) =0 \right\}$ denote the boundary of $T_B.$ Then $$\widehat{T_B}:=T_B\cup \partial T_B\cup \left\{ \infty \right\}$$ is the one-point compactification of $T_B.$

 Also, let $\partial \widehat{T_B}:=\partial T_B\cup \left\{ \infty \right\} .$ Thus, $z\rightarrow \partial \widehat{T_B}$ 
means $\rho \left( z \right) \rightarrow 0$ or $\left| z \right|\rightarrow \infty .$ 

We denote by $C_0\left( T_B \right) $ the space of complex-valued continuous functions $f$ on $T_B$ such that $f\left( z \right) \rightarrow 0$ as $z\rightarrow \partial \widehat{T_B}.$

For a positive Borel measure $\mu$ on $T_B$, we define a function $\tilde{\mu}$ on $T_B$ by \[\tilde{\mu}\left( z \right) :=\int_{T_B}{\left| k_z\left( w \right) \right|^2d\mu \left( w \right)}, \quad  z\in T_B,\] where, for fixed $z\in T_B$, \[k_z(w):= K(z,w)/\sqrt{K(z,z)}, \quad w\in T_B.\]

For $z\in T_B$ and $r>0$, we define the averaging function \[\hat{\mu}_r\left( z \right) :=\frac{\mu \left( D\left( z,r \right) \right)}{V_{\alpha}\left( D\left( z,r \right) \right)}.\]

Let $\mu$ be a positive Borel measure  on $T_B$ and $p>0$. We say that $\mu$ is Carleson measure for the Bergman Space $A_\alpha^{p}(T_B)$, if there exists a positive constant $C > 0$ such that
$$
\int_{T_B} \left| f(z) \right|^p \, d\mu(z) \leq C \lVert f \rVert_{p,\alpha}^p
$$
for every $f \in A_\alpha^{p}(T_B)$.

A positive Borel measure $\mu$ as a vanishing Carleson measure, if for any bounded sequence ${f_k}$ in $A_\alpha^{p}(T_B)$ that converges uniformly to $0$ on every compact subset of $T_B$, we have \[\lim_{k\rightarrow \infty} \int_{T_B}{\left| f_k \right|^p d\mu}=0.\]

First, we review the Bergman metric on domains in $\mathbb{C}^n$. Let $K(z,w)$ be the kernel of $T_{B}$. We define the complex matrix 
$$
\mathbf{B}(z)=(b_{ij}(z))_{1\leq i,j\leq n}=\frac{1}{n+1}\left ( \frac{\partial ^2}{\partial \bar{z_i}\partial z_j}\ln{K(z,z)}\right )_{1\leq i,j\leq n}
$$
as the Bergman matrix of $T_{B}$. 

For a smooth curve $\gamma: \left[ 0,1 \right] \rightarrow T_B$, we define 
$$l(\gamma )=\int_{0}^{1}\left \langle \mathbf{B}(\gamma (t))\gamma '(t),\gamma '(t)\right \rangle dt.$$

Based on the definition of  $l(\gamma )$, we can define the Bergman metric $\beta$ on $T_{B}$ as follows:
$$\beta (z,w)=\inf \{l(\gamma ):\gamma (0)=1,\ \gamma (1)=w\}.$$

Let $D\left( z,r \right) $ denote the Bergman metric ball at $z$ with radius $r$, that is 
 \[D(z,r)=\left\{w \in T_B:  \beta(z,w)<r\right\}.\]

We will use the important  transform $\varPhi :\mathbb{B}_{n}\rightarrow T_{B}$ given by
$$\varPhi(z)=\left ( \frac{\sqrt{2}z'}{1+z_{n}}, \ i\frac{1-z_n}{1+z_n}-i\frac{z' \cdot z'}{(1+z_n)^2} \right ), \quad z\in \mathbb{B}_n$$
and it is not hard to calculate that 
$$\varPhi^{-1}(w)=\left ( \frac{2iw'}{i+w_{n}+\frac{i}{2}w'\cdot w'},\  \frac{i-w_n-\frac{i}{2}w'\cdot w'}{i+w_{n}+\frac{i}{2}w'\cdot w'} \right ), \quad w\in T_{B}.$$
The mapping $\varPhi$ is a biholomorphic map from $\mathbb{B}_n$ to $T_B$ and also a key tool for this paper.

In Krantz's book \cite{krantz2001function}, there is the following proposition \cite[proposition 1.4.12]{krantz2001function}:

Let $\varOmega _1,\ \varOmega _2\subseteq \mathbb{C}^n$ be domains and $ f:\varOmega _1\rightarrow \varOmega _2$ a biholomorphic mapping. Then $f$ induces an isometry of Bergman metrics:$$\beta _{\varOmega _1}\left( z,w \right) =\beta _{\varOmega _2}\left( f\left( z \right) ,f\left( w \right) \right) $$ for all $z,w\in \varOmega _1$.

Hence, taking $\varOmega _1=T_B$ and $\varOmega _2=\mathbb{B}_n$, we have:$$
\beta _{T_B}\left( z,w \right) =\beta _{\mathbb{B}_n}\left( \Phi ^{-1}\left( z \right) ,\Phi ^{-1}\left( w \right) \right) =\tanh ^{-1}\left( \left| \varphi _{\Phi ^{-1}\left( z \right)}\left( \Phi ^{-1}\left( w \right) \right) \right| \right) .
$$
A computation shows that
$$
\beta_{T_B}(z, w)=\tanh ^{-1} \sqrt{1-\frac{\rho(z) \rho(w)}{|\rho(z, w)|^2}} .
$$
Throughout the paper we use C to denote a positive constant whose value may change from line to line but does not depend on the functions being considered. The notation $A\lesssim B$ means that there is a positive constant C such that $A\le CB$, and the notation $A\simeq B$ means that $A\lesssim B$ and $B\lesssim A$.
\section{Main lemmas}
\ \ \ \
To prove our main results, we need the following key lemmas, where Lemmas \ref{lem duliang}-\ref{lem yingshe de xingzhi} are from \cite{jiaxin2023bergman}. They play a crucial role as instrumental lemmas in the present paper.

\begin{lemma}\label{lem duliang}
	There exists a positive integer $N$ such that for any $0<r\le1$ we can find a sequence $\{a_k\}$ in $T_B$ with the following properites:\\
	
	$(1)$ $T_B=\bigcup_{k=1}^{\infty}{D\left( a_k,r \right)};$
	
	$(2)$ The sets $D(a_k,r/4)$ are mutually disjoint;
	
	$(3)$Each point $z\in T_B$ belongs to at most $N$ of the sets $D(a_k,2r)$.
	
\end{lemma}

\begin{lemma}\label{lem dengjia}
	For any $r>0$, the inequalities 
	$$\left| \rho \left( z,u \right) \right|\simeq \left| \rho \left( z,v \right) \right|$$
	hold for all $z,u,v\in T_{B} $ with $\beta (u,v) <r .$
\end{lemma}

\begin{lemma}\label{lem ceByuanpan}
	For any $z\in T_B$ and $r>0$ we have \[V_{\alpha}\left( D\left( z,r \right) \right) \simeq \rho \left( z \right) ^{n+\alpha +1}.\]
\end{lemma}

\begin{lemma}
	Let $a,b$ and $c\in \mathbb{R},$ $1\leqslant p<\infty$ and \[Tf\left( z \right) =\rho \left( z \right) ^a\int_{T_B}{\frac{\rho \left( w \right) ^b}{\rho \left( z,w \right) ^c}f\left( w \right) dV\left( w \right)}\] and \[Sf\left( z \right) =\rho \left( z \right) ^a\int_{T_B}{\frac{\rho \left( w \right) ^b}{\left| \rho \left( z,w \right) \right|^c}f\left( w \right) dV\left( w \right)}.\] Then the following conditions are equivalent for any real $\alpha$.
	
	$(1)$ The operator $S$ is bounded on $L_{\alpha}^{p}\left( T_B \right).$
	
	$(2)$ The operator $T$ is bounded on $L_{\alpha}^{p}\left( T_B \right).$
	
	$(3)$ The parameters satisfy $-p\alpha<\alpha+1<p(b+1)$ and $c=n+\alpha+b+1.$\\
		When $p=\infty$, condition $(3)$ should be $a>0$, $b>-1$, and $c=n+a+b+1$.
	
\end{lemma}

\begin{lemma}\label{lem jifendengshi}
	Let $r, s>0, t>-1,$ and $r+s-t>n+1$, then 
	$$
	\int_{T_{B}}\frac{{\rho (w)}^{t}}{{\rho (z,w)}^{r}{\rho (w,u)}^{s}}dV(w)=\frac{C_{1}(n,r,s,t)}{\rho (z,u)^{r+s-t-n-1}}
	$$
	for all $z,u \in T_{B}$, where
	$$C_{1}(n,r,s,t)=\frac{2^{n+1}{\pi}^{n}\Gamma (1+t)\Gamma (r+s-t-n-1)}{\Gamma (r)\Gamma (s)}.$$
	In particular, let $s,t\in \mathbb{R}$, if $t>-1,s-t>n+1$, then \[\int_{T_B}{\frac{\rho \left( w \right) ^t}{\left| \rho \left( z,w \right) \right|^s}}dV\left( w \right) =\frac{C_1\left( n,s,t \right)}{\rho \left( z \right) ^{s-t-n-1}},\] Otherwise, the above equation is infinity.
\end{lemma}

\begin{lemma}\label{lem yingshe de xingzhi}
	
	The following properties hold for holomorphic mappings from $\mathbb{B}_n$ to $T_B$:	
	
	$(1)$ The real Jacobian of $\varPhi$ at $z\in T_{B}$ is 
	$$
	J_{R}(\varPhi(z))=\frac{2^{n+1}}{{\left | 1+z_n\right |}^{2(n+1)}}.
	$$
	
	$(2)$ The real Jacobian of $\varPhi^{-1}$ at $z\in T_{B}$ is 
	$$
	(J_{R} \varPhi^{-1})(z)=\frac{1}{4\left| \rho (z,\mathbf{i} )\right|^{2(n+1)}}.
	$$
	
	$(3)$  The identity 
	$$
	1-\left \langle \varPhi^{-1}(z),\varPhi^{-1}(w)\right \rangle= \frac{\rho (z,w)}{\rho (z,\mathbf{i} )\rho (\mathbf{i},w)}
	$$
	holds for all $z,w\in T_{B}$, where $\mathbf{i}=(0',i).$ \\ And moreover, 
	$$1- {\left |\varPhi^{-1}(z)\right |}^2= \frac{\rho (z)}{{\left | \rho (z,\mathbf{i})\right |}^2}, 1+[\varPhi^{-1}(z)]_{n} = \frac{1}{\rho (z,\mathbf{i})}.$$
	
	$(4)$ The identity 
	$$
	\rho (z,w)= \rho (\varPhi (\xi ),\varPhi (\eta ))= \frac{1-\left \langle \xi ,\eta \right \rangle}{(1+\xi_{n})(1+\eta_n)}
	$$
	holds for all $z,w\in T_{B},$ where $\xi =\varPhi^{-1}(z), \eta =\varPhi^{-1}(w).$
	
\end{lemma}

\begin{proof}
	The properties mentioned above are derived from \cite{jiaxin2023bergman}, and here we only present the unproven property $(2)$.
	
	Simple calculations show that $$\rho (z,\mathbf{i})=\frac{1}{4}(z'^2-2iz_n+2) ,$$
	so 
	$$
	\begin{aligned}
	\left( J_C\varPhi ^{-1} \right) \left( z \right) &=\left| \begin{matrix}
	\frac{2\rho \left( w,\mathbf{i} \right) -w_{1}^{2}}{2\rho \left( w,\mathbf{i} \right) ^2}&		\frac{-w_1w_2}{2\rho \left( w,\mathbf{i} \right) ^2}&		\cdots&		\frac{-w_1w_{n-1}}{2\rho \left( w,\mathbf{i} \right) ^2}&		\frac{iw_1}{2\rho \left( w,\mathbf{i} \right) ^2}\\
	\frac{-w_2w_1}{2\rho \left( w,\mathbf{i} \right) ^2}&		\frac{2\rho \left( w,\mathbf{i} \right) -w_{2}^{2}}{2\rho \left( w,\mathbf{i} \right) ^2}&		\cdots&		\frac{-w_2w_{n-1}}{2\rho \left( w,\mathbf{i} \right) ^2}&		\frac{iw_2}{2\rho \left( w,\mathbf{i} \right) ^2}\\
	\vdots&		\vdots&		\ddots&		\vdots&		\vdots\\
	\frac{-w_{n-1}w_1}{2\rho \left( w,\mathbf{i} \right) ^2}&		\frac{-w_{n-1}w_2}{2\rho \left( w,\mathbf{i} \right) ^2}&		\cdots&		\frac{2\rho \left( w,\mathbf{i} \right) -w_{n-1}^{2}}{2\rho \left( w,\mathbf{i} \right) ^2}&		\frac{iw_{n-1}}{2\rho \left( w,\mathbf{i} \right) ^2}\\
	\frac{-w_1}{2\rho \left( w,\mathbf{i} \right) ^2}&		\frac{-w_2}{2\rho \left( w,\mathbf{i} \right) ^2}&		\cdots&		\frac{-w_{n-1}}{2\rho \left( w,\mathbf{i} \right) ^2}&		\frac{i}{2\rho \left( w,\mathbf{i} \right) ^2}\\
	\end{matrix} \right|\\
	&=\frac{i}{2\rho \left( z,\mathbf{i} \right)^{n+1}}.
	\end{aligned}
	$$
	Therefore $J_R\left( \varPhi \left( z \right) \right) =\left| J_C\left( \varPhi \left( z \right) \right) \right|^2=\frac{1}{ 4\left| \rho \left( z,\mathbf{i} \right) \right|^{2\left( n+1 \right)}}
	.$
\end{proof}
\\

The following lemma plays a key role in estimating the inequality.
\begin{lemma}\label{lem qu bian liang}
	For any $z,w\in {T_{B}}$, we have 
	$$
	2\left| \rho \left( z,w \right) \right|\ge \max \left\{ \rho \left( z \right) ,\rho \left( w \right) \right\}.
	$$
\end{lemma}

\begin{proof}
	Let $\delta_{t}(w)=(tw',t^{2}w_{n})$ be the nonisotropic dilation for $t>0$, $w\in T_{B}$. 
	
	For any $w \in T_B$ and each fixed $z \in T_B$, consider the holomorphic mapping $$h_z\left( w \right) :=\left( w'-z',w_n-\text{Re}z_n-iw'\overline{z'}+\frac{i\left| z' \right|^2}{2}+\frac{i\overline{z'}\cdot \overline{z'}}{4}+\frac{iw'\cdot z'}{4} \right).$$ 
	
	It is evident that $h_z(w)$ is a holomorphic automorphism of $T_{B}$. Thus, the mapping $\sigma _{z}:=\delta _{\rho (z)^{-\frac{1}{2}}}\circ h_{z}$ is a holomorphic automorphism of $T_{B}$. Simple calculations reveal that $\sigma _{z}(z)=\mathbf{i}:=(0',i)$ and $$(J_{\mathbb{C}}\sigma_{z})(w)=\rho (z)^{-\frac{n+1}{2}},$$ where $(J_{\mathbb{C}}\sigma_{z})(w)$ denotes the complex Jacobian of $\sigma_{z}$ at $w$. 
	
	Hence, we obtain 
	$$
	\begin{aligned}    
	K_{\alpha}\left( z,w \right)& =\left( J_{\mathbb{C}}\sigma _z \right) \left( z \right) K_{\alpha}\left( \sigma _z\left( z \right) ,\sigma _z\left( w \right) \right) \overline{\left( J_{\mathbb{C}}\sigma _z \right) \left( w \right) }
	\\
	&=K_{\alpha}\left( \mathbf{i},\sigma _z\left( w \right) \right) \,\,\rho \left( z \right) ^{-n-1}.
	\end{aligned}
	$$

	It's worth noting that 
	$$
	\left| K_{\alpha}\left( \mathbf{i},w \right) \right|\le \frac{\varGamma \left( n+\alpha +1 \right) 2^{\alpha}}{\pi ^n\varGamma \left( \alpha +1 \right)}
	$$
	for all $w\in T_{B}$. By substituting the expression for $K_{\alpha}\left( z,w \right)$ into the above inequality, considering the arbitrariness of $\alpha$ and $n$, and rearranging the positions of $z$ and $w$, we arrive at the desired result.
\end{proof}
\\

The following lemma \ref{lem ci tiao he fang da} is commonly encountered in the operator theory on function spaces, illustrating that the growth of functions in Bergman spaces is controlled.
\begin{lemma}\label{lem ci tiao he fang da}
	On the Bergman space $A_{\alpha}^{p}\left( T_B \right)$, every valuation functional is a bounded linear functional. Specifically, for each function $f\in T_B$, we have 
	\[
	\left| f\left( z \right) \right|^p \leq \frac{C}{\rho \left( z \right) ^{n+\alpha+1}}\int_{D\left( z,r \right)}{\left| f\left( w \right) \right|^pdV_\alpha\left( w \right)},
	\]
	where $0<p<\infty$, $\alpha>-1$, $r>0$ and $C$ is a positive constant .
\end{lemma}
	\begin{proof}
		Let $f\in H\left( T_B \right)$. Then $f\circ\varPhi \in H\left(\mathbb{B}_n \right)$. Note that $D\left( \mathbf{i},r \right) =\varPhi \left( B\left( 0,R \right) \right)$ with $R=\tanh \left( r \right)$. By the subharmonicity of $\left| f \right|^p$ and variable transformation, we have 
		\[
		\begin{aligned}
		\left| f\left( \varPhi \left( 0 \right) \right) \right|^p &\leq \frac{1}{V_{\alpha}\left( D\left( 0,R \right) \right)}\int_{B\left( 0,R \right)}{\left| f\left( \varPhi \left( \xi \right) \right) \right|^p}dV_{\alpha}\left( \xi \right)\\
		&= \frac{1}{V_{\alpha}\left( D\left( 0,R \right) \right)}\int_{D\left( \mathbf{i},r \right)}{\left| f\left( w \right) \right|^p}\frac{1}{\left| \rho \left( w,\mathbf{i} \right) \right|^{2\left( n+\alpha +1 \right)}}dV_{\alpha}\left( w \right).\\
		\end{aligned}
		\]
		
		Since $f\left( \varPhi \left( 0 \right) \right) =f\left( \mathbf{i} \right)$ and $2\rho \left( w,\mathbf{i} \right) \geq 1$, there exists a positive constant $C$ such that 
		\[
		\left| f\left( \mathbf{i} \right) \right|^p \leq C\int_{D\left( \mathbf{i},r \right)}{\left| f\left( w \right) \right|^pdV_{\alpha}\left( w \right)}.
		\] 
		Replacing $f$ by $f\circ \sigma _{z}^{-1}$ in the above inequality, we obtain
		\begin{equation*}
		\left| f\left( z \right) \right|^p \leq \frac{C}{\rho \left( z \right) ^{n+\alpha +1}}\int_{D\left( z,r \right)}{\left| f\left( w \right) \right|^pdV_{\alpha}\left( w \right)}.
		\end{equation*}
		This completes the proof of the lemma.
	\end{proof}
\\

As is well known, the reproducing kernel is an element of the Bergman space $A_\alpha^p(T_B)$, hence its norm is finite. Lemma \ref{lem zai shng he de da xiao} below clarifies the size of the reproducing kernel.

\begin{lemma}\label{lem zai shng he de da xiao}
	For $1<p<\infty$ and $\alpha>-1$, for each $z\in T_B$, the Bergman kernel function $K_{\alpha,z}(w)$ belongs to $A_\alpha^p(T_B)$, and its norm is $C\rho \left( z \right) ^{-\left( n+\alpha+1 \right) /p'}$, where $p'$ is the conjugate exponent of $p$ and $C$ is a positive constant depending only on $n,\alpha$ and $p$.
\end{lemma}

\begin{proof}
	The above conclusion can be obtained through straightforward calculation and the lemma \ref{lem jifendengshi}, so we omit the proof.
\end{proof}
\\

The following lemma \ref{lem ruo shou lian de deng jia tiao jian} provides an equivalent condition for weak convergence of functions in Bergman spaces. Although it is well-known, its proof has not been found, thus, we provide one below.
\begin{lemma}\label{lem ruo shou lian de deng jia tiao jian}
	Suppose $\left\{ f_j \right\}$ is a sequence in $A_{\alpha}^{p}\left( T_B \right)$ with $1<p<\infty.$ Then $f_j\rightarrow 0$ weakly in $A_{\alpha}^{p}\left( T_B \right)$ as $j\rightarrow \infty$ if and only if $\left\{ f_j \right\}$ is bounded in $A_{\alpha}^{p}\left( T_B \right)$ and converges to $0$ uniformly on each compact subset of $T_B$.
\end{lemma} 

\begin{proof}
	(1) We first prove the necessary part of the lemma. 
	
	Suppose $\left\{ f_j \right\}$ converges to $0$ weakly in $A_{\alpha}^{p}\left( T_B \right)$ as $j\rightarrow \infty$, then $\left\{ f_j \right\}$ is bounded in $A_{\alpha}^{p}\left( T_B \right)$ according to the uniform boundedness principle.
	
	 Based on the above lemma \ref{lem ci tiao he fang da}, $\left\{ f_j \right\}$ is uniformly bounded on every compact subset of $A_{\alpha}^{p}\left( T_B \right)$ and thus is a normal family.
	 
	  Note that $\left\{ f_j \right\}$ converges to $0$ pointwise according to the property of the reproducing kernel as $j\rightarrow \infty$, so $\left\{ f_j \right\}$ converges uniformly to $0$ on every compact subset of $T_B$ as $j\rightarrow \infty$.
	
	(2) Now we prove the sufficiency part of the lemma.
	
	 Suppose $\sup_j\lVert f_j \rVert _{p,\alpha}<\infty $ and $\left\{ f_j \right\}$ converges uniformly to $0$ on every compact subset of $T_B$ as $j\rightarrow \infty$. 
	For any $\varepsilon >0$ and any $g\in A_{\alpha}^{q}\left( T_B \right)$ (where $\frac{1}{p}+\frac{1}{q}=1$), there exists a compact subset $K$ of $T_B$ such that 
	\[
	\left( \int_{T_B\backslash K}{\left| g\left( z \right) \right|^q\rho \left( z \right) ^{\alpha}dV\left( z \right)} \right) ^{\frac{1}{q}}<\frac{\varepsilon}{2M}
	.\] Then, through the Holder inequality, we have
	\begin{equation*}
	\begin{aligned}
	\left| \left( f_j,g \right) \right| &\leq \left| \int_K{f_j\left( z \right) \overline{g\left( z \right) }\rho \left( z \right) ^{\alpha}dV\left( z \right)} \right|\\
	&+\left| \int_{T_B\backslash K}{f_j\left( z \right) \overline{g\left( z \right) }\rho \left( z \right) ^{\alpha}dV\left( z \right)} \right|\\
	&\leq V\left( k \right) ^{\frac{1}{p}}\lVert g\left( z \right) \rVert _{q,\alpha}\mathop{sup}_{z\in K}\left| f_j\left( z \right) \right|\\
	&+\lVert f_j\left( z \right) \rVert _{p,\alpha}\left( \int_{T_B\backslash K}{\left| g\left( z \right) \right|^q\rho \left( z \right) ^{\alpha}dV\left( z \right)} \right) ^{\frac{1}{q}}.
	\end{aligned}
	\end{equation*}
	Since $\left\{ f_j \right\}$ uniformly converges to $0$ on $K$ as $j\rightarrow \infty$, there exists $N>0$ such that when $j\ge N$, the second term of the above formula is less than $\varepsilon/2$.
	
	Therefore, $\left( f_j,g \right)\rightarrow 0$ holds when $j\rightarrow \infty $, and then $\left\{ f_j \right\}$ converges weakly to $0$ in $A_{\alpha}^{p}\left( T_B \right)$ as $j\rightarrow \infty$. The proof is complete.
\end{proof}

\begin{lemma}\label{lem zitui}
	As $\left| z \right|$ approaches infinity, $\rho \left( z,\mathbf{i} \right)$ also tends to infinity.
\end{lemma}

\begin{proof}
	We consider two cases.
	
	Case 1: When $z_n$ is fixed and $\left| z_k \right| \rightarrow \infty$ for $k = 1, 2, \ldots, n-1$, we have
	\[
	\begin{aligned}
	4\left| \rho \left( z,\mathbf{i} \right) \right| &\ge 4\text{Re}\rho \left( z,\mathbf{i} \right) \\
	&=\sum_{k=1}^{n-1}{\text{Re}z_{k}^{2}}+2\text{Im}z_n+2 \\
	&\ge \sum_{k=1}^{n-1}{\text{Re}z_{k}^{2}}+2\sum_{k=1}^{n-1}{\left( \text{Im}z_k \right) ^2}+2 \\
	&=\sum_{k=1}^{n-1}{\left( \text{Re}z_k \right) ^2}+\sum_{k=1}^{n-1}{\left( \text{Im}z_k \right) ^2}+2 \\
	&=\sum_{k=1}^{n-1}{\left| z_k \right|^2}+2.
	\end{aligned}
	\]
This implies $\left| \rho \left( z,\mathbf{i} \right) \right| \rightarrow \infty$ as $\left| z_k \right| \rightarrow \infty$ for $k = 1, 2, \ldots, n-1$.

	Case 2: When the previous $n-1$ variables are fixed (assuming they are all zero without loss of generality), then
	\[
	2\left| \rho \left( z,\mathbf{i} \right) \right|\ge\left| z_n \right|-1,
	\]	
	which shows that $\left| \rho \left( z,\mathbf{i} \right) \right| \rightarrow \infty$ as $\left| z_n \right| \rightarrow \infty$.
	
	Therefore, from cases 1 and 2, we can obtain that $\left| \rho \left( z,\mathbf{i} \right) \right| \rightarrow \infty$ as $\left| z \right| \rightarrow \infty$, which completes the proof.
\end{proof}\\

In the Bergman spaces on the unit disk, it is known that as $\left| z \right|$ approaches the boundary, the normalized reproducing kernel weakly converges to 0 in $A_{\alpha}^{p}\left( T_B \right)$, as referenced in \cite{zhu2007operator}. The following lemma \ref{lem ruo shou lian}  informs us that in the Bergman spaces on  tubes, the normalized reproducing kernel also possesses this property.

\begin{lemma}\label{lem ruo shou lian}
	For $1<p<\infty$ and $\alpha>-1$, we have $K_{\alpha,z}\lVert K_{\alpha,z} \rVert _{p,\alpha}^{-1}\rightarrow 0$ weakly in $A_{\alpha}^{p}\left( T_B \right)$ as $z\rightarrow \partial \widehat{T_B}.$ 
\end{lemma}
	\begin{proof}
		According to Lemma \ref{lem ruo shou lian de deng jia tiao jian}, our proof is to demonstrate the uniform convergence to $0$ of $K_{\alpha,z}\lVert K_{\alpha,z} \rVert _{{p,\alpha}}^{-1}$ on each $Q_j:= \overline{D\left( \mathbf{i},j \right) }$.

		By applying Lemma \ref{lem zai shng he de da xiao} and Lemma \ref{lem dengjia}, we can establish the existence of a constant $C>0$ such that 
		\[
		\sup_{w\in Q_j}\left| \frac{K_{\alpha,z}\left( w \right)}{\lVert K_{\alpha,z} \rVert _{p,\alpha}} \right|\le C\frac{\rho \left( z \right) ^{(n+\alpha+1)/p'}}{\left| \rho \left( z,\mathbf{i} \right) \right|^{n+\alpha+1}}
		\] 
		for all $z\in T_B$. Given that $2\rho \left( z,\mathbf{i} \right) =\frac{1}{2}\left| z'^2-2iz_n+2 \right|\ge1$ for all $z\in T_B$, we deduce 
		\[
		\sup_{w\in Q_j}\left| \frac{K_{\alpha,z}\left( w \right)}{\lVert K_{\alpha,z}\rVert _{p,\alpha}} \right|\le C\rho \left( z \right) ^{(n+\alpha+1)/p'},
		\]
		which indicates that $K_{\alpha,z}\lVert K_{\alpha,z}\rVert _{{p,\alpha}}^{-1}\rightarrow 0$ uniformly on $Q_j$ as $z\rightarrow \partial T_B$. 
		
		Furthermore, through the utilization of Lemma \ref{lem qu bian liang} and Lemma \ref{lem zitui}, we ascertain that 
		\[
		\sup_{w\in Q_j}\left| \frac{K_{\alpha,z}\left( w \right)}{\lVert K_{\alpha,z}\rVert _{p,\alpha}} \right|\le\frac{C}{\left| \rho \left( z,\mathbf{i} \right) \right|^{\left( n+\alpha+1 \right) /p}}\rightarrow 0,
		\]
		which demonstrates that $K_{\alpha,z}\lVert K_{\alpha,z}\rVert _{{p,\alpha}}^{-1}\rightarrow 0$ uniformly on $Q_j$ as $\left| z \right|\rightarrow \infty $. Thus, the lemma is successfully proven.
	\end{proof}

\section{Carleson measures on $T_B$}
\ \ \ \
As mentioned in the introduction, the importance of Carleson measures in studying the analytical properties of Toeplitz operators is significant. In the following, we will introduce two important theorems related to Carleson measures, which provide powerful support for the proof of the main results theorem \ref*{mianth1} and \ref*{mianth2}.
 
The following theorem \ref*{thCraleson} presents some equivalent characterizations of Carleson measures on tubes, and this theorem is derived from \cite{jiaxin2023bergman}.
\begin{theorem}\cite{jiaxin2023bergman}\label{thCraleson}
Suppose $0<p<\infty,r>0,\alpha>-1$ and $\mu$ is a positive Borel measure on the $T_B.$ Then the following conditions are equivalent.\\

	$(1)$ $\mu$ is a Carleson measure for $A_{\alpha}^{p}\left( T_B \right).$

	$(2)$ There exists a constant $C>0$ such that \[\int_{T_B}{\frac{\rho \left( z \right) ^{n+\alpha+1}}{\left| \rho \left( z,w \right) \right|^{2\left( n+\alpha+1 \right)}}d\mu \left( w \right)}\le C\] for all $z\in T_B$

	$(3)$ There exists a constant $C>0$ such that \[\mu \left( D\left( z,r \right) \right) \le C\rho \left( z \right) ^{n+\alpha+1}\] for all $z\in T_B.$

	$(4)$ There exista a constant $C>0$ such that \[\mu \left( D\left( a_k,r \right) \right) \le C\rho \left( a_k \right) ^{n+\alpha+1}\] for all $k\ge1,$ where $\left\{ a_k \right\} $ is an $r$-lattice in the Bergman metric.

\end{theorem}

The proof of Theorem \ref{thCraleson} has been provided in \cite{jiaxin2023bergman}, and we will not repeat it here.\\

Following that, we provide some equivalent characterizations of vanishing Carleson measures. These characterizations illustrate the property of being a vanishing Carleson measure for $A_{\alpha}^{p}\left( T_B \right)$, which depends neither on $p$ nor on $r$ in the following theorem \ref*{thVanishing Carleson}.

\begin{theorem}\label{thVanishing Carleson}
	Suppose $1<p<\infty,r>0,\alpha>-1 $ and $\mu$ is a positive Borel measure on the $T_B.$ Then the following conditions are equivalent.\\

		$(1)$ $\mu$ is a vanishing Carleson measure for $A_{\alpha}^{p}\left( T_B \right).$
		
		$(2)$ The measure $\mu$ satisfies \[\lim_{a\rightarrow \partial \widehat{T_B}} \int_{T_B}{\frac{\rho \left( a \right) ^{n+\alpha+1}}{\left| \rho \left( z,a \right) \right|^{2\left( n+\alpha+1 \right)}}d\mu \left( z \right)}=0.\]

		$(3)$ The measure $\mu$ has the property that \[\lim_{a\rightarrow \partial \widehat{T_B}} \frac{\mu \left( D\left( a,r \right) \right)}{\rho \left( a \right) ^{n+\alpha+1}}=0.\]

		$(4)$ For $\{a_k\}$ an $r$-lattice in the Bergman metric, we have \[\lim_{k\rightarrow \infty} \frac{\mu \left( D\left( a_k,r \right) \right)}{\rho \left( a_k \right) ^{n+\alpha+1}}=0.\]
\end{theorem}

\begin{proof}
	
	$(1) \Rightarrow (2)$: According to Lemmas \ref{lem zai shng he de da xiao} and \ref{lem ruo shou lian}, take 
	\[g_a(z)=\left( \frac{\rho \left( a \right) ^{n+\alpha +1}}{\left| \rho \left( z,a \right) \right|^{2\left( n+\alpha +1 \right)}} \right) ^{1/p}.\]
	By Lemma  \ref{lem zai shng he de da xiao} and \ref{lem ruo shou lian}, we see that $g_a(z)$ converges weakly to $0$ in $A_\alpha^p\left( T_B \right)$ as $a\rightarrow \partial \widehat{T_B}$. Therefore,
	\[\int_{T_B}{\frac{\rho \left( a \right) ^{n+\alpha+1}}{\left| \rho \left( z,a \right) \right|^{2\left( n+\alpha+1 \right)}}d\mu \left( z \right)}=C\int_{T_B}{\left| g_a\left( z \right) \right|^pd\mu \left( z \right)}\rightarrow 0\] 
	as $a\rightarrow \partial \widehat{T_B}$, which shows that  (2)  holds.\\
	
	$(2) \Rightarrow (3)$:
	Obviously, the following equation holds:
	\[\lim_{a\rightarrow \partial \widehat{T_B}} \int_{T_B}{\frac{\rho \left( a \right) ^{n+\alpha+1}}{\left| \rho \left( z,a \right) \right|^{2\left( n+\alpha+1 \right)}}d\mu \left( z \right)}=0.\] 
	By Lemma \ref{lem dengjia}, $\left| \rho \left( z,a \right) \right|$ and $\rho \left( z \right) $ are comparable when $z\in D\left(a,r\right)$. Therefore, $\mu$ has that protery.\\
	
	$(3) \Rightarrow (4)$:
	We know that $a_k\rightarrow \partial \widehat{T_B}$ as $k\rightarrow \infty $ if $\{a_k\}$ is an $r$-lattice in the Bergman metric. So the conclusion is trivial.\\
	
	$(4) \Rightarrow (1)$:
	If the equality holds, we show that the inclusion map $i_p$ from $A_\alpha^p\left( T_B \right)$ into $L^p(T_B,d\mu)$ is compact. 
	
	To this end, we assume that $\{f_j\}$ is a sequence in $A_\alpha^p\left( T_B \right)$ that converges to $0$ uniformly on compact subsets of $T_B$ and $\lVert f_j \rVert _{p,\alpha}\le M$ for some positive constant $M$.
	
	 By assumption, given $\varepsilon >0$ there exists a positive integer $N_0$ such that \[\frac{\mu \left( D\left( a_k,r \right) \right)}{\rho \left( a_k \right) ^{n+\alpha+1}}<\varepsilon, \    k\ge N_0.\]
By Lemma \ref{lem duliang}, there is a constant $C>0$ such that 
$$
	\begin{aligned}
	&\sum_{k=N_0}^{\infty}{\int_{D\left( a_k,r \right)}{\left| f_j\left( z \right) \right|^pd\mu \left( z \right)}}\\
	&\le C\sum_{k=N_0}^{\infty}{\frac{\mu \left( D\left( a_k,r \right) \right)}{\rho \left( a_k \right) ^{n+\alpha+1}}}\int_{D\left( a_k,2r \right)}{\left| f_j\left( z \right) \right|^pdV_\alpha\left( z \right)}
	\\
	&\le \varepsilon CN\int_{T_B}{\left| f_j\left( z \right) \right|^pdV_\alpha\left( z \right)}\le \varepsilon CNM^p
	\end{aligned}
$$
	for all $j$, where $C,N$ and $M$ are all independent of $\varepsilon$. Since \[\lim_{j\rightarrow \infty} \sum_{k=1}^{N_0-1}{\int_{D\left( a_k,r \right)}{\left| f_j\left( z \right) \right|^pd\mu \left( z \right)}}=0,\] 
	we have 
	$$
	\begin{aligned}
	\lim_{j\rightarrow \infty} &sup\int_{T_B}{\left| f_j\left( z \right) \right|^pd\mu \left( z \right)}
	\\
	&\le \lim_{j\rightarrow \infty} sup\left[ \sum_{k=1}^{N_0-1}{\int_{D\left( a_k,r \right)}{\left| f_j\left( z \right) \right|^pd\mu \left( z \right)}}+\sum_{k=N_0}^{\infty}{\int_{D\left( a_k,r \right)}{\left| f_j\left( z \right) \right|^pd\mu \left( z \right)}} \right] 
	\\
	&\le \varepsilon CNM^p.
	\end{aligned}
	$$
	Since $\varepsilon$ is arbitrary, we know that $\mu$ is a vanishing Carleson measure. The proof of the theorem is complete.	

\end{proof}

\section{Dense subspaces of $A_\alpha^p\left( T_B \right)$}
\ \ \ \
For the case of unit disk $\mathbb{D}$, we know that the space $H^\infty{(\mathbb{D})}$ of bounded analytic functions  is dense in the Bergman spaces $A_\alpha^p(\mathbb{D})$. Since the Toeplitz operator is an integral operator, it is always well-defined, making the Toeplitz operator densely defined in $A_\alpha^p(\mathbb{D})$. 

However, in the case of tubes, there is not dense property as good as in a unit disk. In order to study the theory of Toeplitz operators on the Bergman space $A_\alpha^p\left( T_B \right)$ over tubular domains $T_B$,  we have to first find a dense subspaces of $A_\alpha^p\left( T_B \right)$, and then obtain the Toeplitz operators densely defined on $A_\alpha^p\left( T_B \right)$.

 To do this, we first introduce the following definition: 

Let $\mathcal{M}_+$ be the set of all positive Borel measure $\mu$ such that \[\int_{T_B}{\frac{d\mu \left( z \right)}{\left| \rho \left( z,\mathbf{i} \right) \right|^t}}<\infty \] for some $t>0$.

Given $t \in \mathbb{R}$, we denote by $S_t$ the vector spaces of functions $f$ holomorphic in $T_B$ satisfying \[\mathop{sup}_{z\in T_B}\left| \rho \left( z,\mathbf{i} \right) \right|^t\left| f\left( z \right) \right|<\infty .\]

\begin{theorem}
	If $1\le p<\infty$, $\alpha>-1$, $t>n+(\alpha+1)/p$ and $\mu \in \mathcal{M}_+$, then the Toeplitz
	operator $T_\mu$ is densely defined on $A_\alpha^p(T_B)$.
\end{theorem}
	\begin{proof}
		To prove this theorem, we need to find a dense subspace $S_t$ of $A_\alpha^p(T_B)$, such that for any $f\in S_t$, the following equation holds: $$\int\limits_{T_B}{|}K_{\alpha}\left( z,w \right) f\left( w \right) |d\mu \left( w \right) <\infty .$$
		Now we prove the density of $S_t$ in $A_\alpha^p(T_B)$. Let $f_j=f\cdot \chi_{Q_j}$ for $j=1,2,\cdots$, where $Q_j=\overline{D(\mathbf{i},j)}$ and $\chi_{Q_j}$ is the characteristic function of $Q_j$. Clearly, $\lVert f_j-f \rVert_{p,\alpha}\rightarrow 0$ as $j\rightarrow \infty$.
		Given $\lambda >-1$, let $P_{\lambda}$ be the integral operator given by
		\[P_{\lambda}g(z) = c_{\lambda}\int_{T_B}{\frac{\rho(w)^{\lambda}}{\rho(z,w)^{n+1+\lambda}}g(w)dV(w)}, \ z\in T_B,\]
		where $c_{\lambda}=\frac{2^{n+1+2\lambda}\Gamma(n+1+\lambda)}{\pi^n\Gamma(1+\lambda)}$. It was shown in \cite[Theorem 3.2]{1993integral} that $P_\lambda$ is a bounded projection from $L_\alpha^p(T_B)$ onto $A_\alpha^p(T_B)$, provided that $\lambda>(\alpha+1)/p-1$.
		Taking $\lambda=t-n-1$, by Holder's inequality, we obtain
		\[\left| P_{t-n-1}f_j(z) \right| \le c_{t-n-1}\lVert f \rVert_{p,\alpha}V_{\alpha}(Q_j)^{1/p'}\sup_{w\in Q_j}\frac{\rho(w)^{t-n-\alpha -1}}{\left| \rho(z,w) \right|^t}\]
		for all $z\in T_B$. Combined with Lemma \ref{lem ceByuanpan}, we only need to consider the upper bound of
		\[\sup_{w\in Q_j}\frac{\rho(w)^{t-n-\alpha-1}}{\left| \rho(z,w) \right|^t}.\]
		Since the $Q_j$ are fixed, a simple calculation yields
		\[\left| P_{t-n-1}f_j(z) \right| \le \frac{C\lVert f \rVert_{p,\alpha}}{\left| \rho(z,\mathbf{i}) \right|^{t}}\]
		for all $z\in T_B$, where $C>0$ is a constant independent of $z$. \\
		Thus $$P_{t-n-1}f_j(z) \in S_{t}.$$
		Since $f\in A_\alpha^p(T_B)$ and $P_{\alpha-n-1}$ is a bounded projection from $L_\alpha^p(T_B)$ onto $A_\alpha^p(T_B)$, we have
		\[\lVert P_{\alpha -n-1}f_j(z) -f \rVert_{p,\alpha}=\lVert P_{\alpha -n-1}\left( f_j(z) -f \right) \rVert_{p,\alpha}\le \lVert P_{\alpha -n-1} \rVert \lVert f_j-f \rVert_{p,\alpha}\rightarrow 0\]
		as $j\rightarrow\infty$. This implies that $S_t$ is dense in $A_\alpha^p(T_B)$.

		According to Lemma \ref{lem qu bian liang}, we know that there exists a constant $C>0$ such that 
		$$ 
		\begin{aligned}
			\int\limits_{T_B}{|}K_{\alpha}\left( z,w \right) f\left( w \right) |d\mu \left( w \right) &\lesssim \frac{1}{\rho \left( z \right) ^{n+\alpha +1}}\int\limits_{T_B}{\frac{\left| \rho \left( w,\mathbf{i} \right) \right|^t\left| f\left( w \right) \right|}{\left| \rho \left( w,\mathbf{i} \right) \right|^t}d\mu \left( w \right)}\\
			&\lesssim \frac{1}{\rho \left( z \right) ^{n+\alpha +1}}\int\limits_{T_B}{\frac{d\mu \left( w \right)}{\left| \rho \left( w,\mathbf{i} \right) \right|^t}}\\
			&<\infty 
		\end{aligned}
		$$ 
		holds. This completes the proof of the theorem.	

	\end{proof}

\section{Characterization of Toeplitz operators}
\ \ \ \
Building on the foundational results previously, we can now to  prove the following result:
\begin{theorem}\label{mianth1}
	Suppose that $r>0$, $1<p<\infty$, $0<q<\infty$, $\alpha>-1$ and that $\mu\in \mathcal{M}_+$. Then the following conditions are equivalent:\\

		$(1)$ $T_{\mu}$ is bounded on $A_\alpha^p(T_B)$.
		
		$(2)$ $\widetilde{\mu}$ is a bounded function on $T_B$.
		
		$(3)$ $\widehat{\mu}_r$ is a bounded function on $T_B$.
		
		$(4)$ $\mu$ is a Carleson measure for $A_\alpha^q(T_B)$.

\end{theorem}

\begin{proof}
$(1) \Rightarrow (2)$:
We have the well-known conclusion \[\left< T_{\mu}k_z,k_z \right> =\int_{T_B}{\left| k_z\left( w \right) \right|^2d\mu \left( w \right)}=\tilde{\mu}\left( z \right),\] And according to the conditions and lemma \ref{lem zai shng he de da xiao} we have\[\left| \left< T_{\mu}k_z,k_z \right> \right|\le \lVert T_{\mu}k_z \rVert _{p,\alpha}\lVert k_z \rVert _{p',\alpha}\le \frac{\lVert T_{\mu} \rVert \lVert K_z \rVert _{p,\alpha}\lVert K_z \rVert _{p',\alpha}}{K\left( z,z \right)}=C\lVert T_{\mu} \rVert.\] 
Therefore, $\tilde{\mu}(z)$ is a bounded function on $T_B$.\\

$(2) \Rightarrow (3)$:
By the definition of $\hat{\mu}_r$ and Theorem \ref{thCraleson}, the expected results are obtained by the following identity\[V_{\alpha}\left( D\left( z,r \right) \right) =K\rho \left( z \right) ^{n+\alpha +1},\] where K is a constant independent of z.\\

$(3) \Rightarrow (4)$:
This result can be obtained from the fact that the Carleson measure in Theorem \ref{thCraleson} is not dependent on $p$, so it is trivial.\\

$(4) \Rightarrow (1)$:
Since $\mu$ is a Carlson measure for $A_\alpha^q(T_B)$, and according to the Carlson measure does not depend on the index we know $\mu$ is a Carlson measure for $A_\alpha^1(T_B)$. Combining $T_{\mu} f$ is well defined on $S_t$, there exists a constant $C>0$ such that
\begin{align*}
\int\limits_{T_B} |K_\alpha(z,w) f(w)| d\mu(w) ~\leq~& C \int\limits_{T_B} |K_\alpha(z,w) f(w)| dV_\alpha(w)\\
\leq~& C\int\limits_{T_B}{\frac{\rho \left( w \right) ^{\alpha}dV\left( w \right)}{|\rho \left( z,w \right) |^{n+\alpha +1}|\rho \left( w,\mathbf{i} \right) |^t}}.
\end{align*}

According to the holomorphic automorphism from $\mathbb{B}_n$ to $T_B$, let $$w=\varPhi \left( \xi \right),z=\varPhi \left( \eta \right),$$ and make a variable transformation of the above integral. we obtain

$$
\begin{aligned}
&\int\limits_{T_B}{\frac{\rho \left( w \right) ^{\alpha}dV\left( w \right)}{|\rho \left( z,w \right) |^{n+\alpha +1}|\rho \left( w,\mathbf{i} \right)|^t}}\\
&=2^{n+1}\left| 1+\eta _n \right|^{n+\alpha +1}\int\limits_{\mathbb{B}_n}{\frac{\left( 1-\left| \xi \right|^2 \right) ^{\alpha}dV\left(\xi\right)}{\left| 1-\left< \eta ,\xi \right> \right|^{n+\alpha +1}|1+\xi _n|^{n+\alpha +1-t}}}.
\end{aligned}
$$
By \cite[Theorem 3.1]{zhang2018integral}, the last integral is dominated by a constant multiple of
\[\frac{1}{|1+\eta_n|^{n+\alpha+1-t}}\left(1+\log\frac{|1+\eta_n|}{1-|\eta|^2}\right),\]
Thus, there exists a constant $C>0$ such that
\begin{align*}
\int\limits_{T_B}{\frac{\rho \left( u \right) ^{\alpha}dV\left( u \right)}{|\rho \left( z,u \right) |^{n+\alpha +1}|\rho \left( u,\mathbf{i} \right) |^t}}
~\leq~& C |1+\eta_n|^t\left(1+ \log \frac {1}{1-|\eta|^2}\right)\\
~=~& \frac {C}  {|\rho(z,\mathbf{i})|^t} \left(1+\log \frac {|\rho(z,\mathbf{i})|^2}{\rho(z)}\right)
\end{align*}
holds. 

By the elementary inequality $\log x<x^{\varepsilon}$, where $$x>1,0<\varepsilon <\min \{\left( \alpha +1 \right) /p,t-\left( \alpha +n+1 \right) /p\},$$ we have  \[\int\limits_{T_B} |T_{\mu}f(z)|^p dV_\alpha(z) ~\leq~ C\left( \int\limits_{T_B}\frac{dV_\alpha(z)}{|\rho(z,\mathbf{i})|^{pt}} +\int\limits_{T_B}\frac{\rho(z)^{-p\epsilon}}{|\rho(z,\mathbf{i})|^{p(t-2\epsilon)}}dV_\alpha(z)\right).\] This shows that $T_\mu$ is densely defined on $S_t$ and can be extended to be a bounded operator on $A_\alpha^p(T_B)$.

This completes the proof of Theorem \ref{mianth1}.
\end{proof}
\\

The above theorem \ref*{mianth1} characterizes the boundedness of the Toeplitz operator $T_{\mu}$, and now we will characterize the compactness of $T_{\mu}$ in the following theorem \ref*{mianth2}.

\begin{theorem}\label{mianth2}
	Suppose that $r>0$, $1<p,q<\infty$, $\alpha>-1$ and that $\mu\in \mathcal{M}_+$. Then the following conditions are equivalent:\\

		$(1)$ $T_{\mu}$ is compact on $A_\alpha^p(T_B)$.
		
		$(2)$ $\widetilde{\mu}$ belongs to $C_0(T_B)$.
		
		$(3)$ $\widehat{\mu}_r$ belongs to $C_0(T_B)$.
		
		$(4)$ $\mu$ is a vanishing Carleson measure for $A_\alpha^q(T_B)$.

\end{theorem}

\begin{proof}
	$(1) \Rightarrow (2)$:	
	Assume that $T_\mu$ is a compact operator on $A_\alpha^p(T_B)$ for some $p>1$.
	
	Since \[\left| \tilde{\mu}\left( z \right) \right|\le C\lVert T_{\mu}\left( \frac{K_z}{\lVert K_z \rVert _{p,\alpha}} \right) \rVert _{p,\alpha}\] for all $z\in T_B.$ Combining the compactness of $T_\mu$ and lemma \ref{lem ruo shou lian}, implies $\tilde{\mu}\in C_0(T_B)$.\\
	
	$(2) \Rightarrow (3)$:
	By the definition of $\hat{\mu}_r$ and Theorem \ref{thVanishing Carleson}, the expected results are obtained by the following identity\[V_{\alpha}\left( D\left( z,r \right) \right) =K\rho \left( z \right) ^{n+\alpha +1},\] where K is a constant independent of z.	\\
	
	$(3) \Rightarrow (4)$:
	This conclusion can be drawn from the fact that the vanishing Carleson measure in Theorem \ref{thVanishing Carleson} is not dependent on $p$, so it is trivial.\\
	
	$(4) \Rightarrow (1)$:
	Since $\mu$ is a vanishing Carlson measure for $A_\alpha^q(T_B)$, and according to the vanishing Carlson measure does not depend on the index we know $\mu$ is a vanishing Carlson measure for $A_{\alpha}^{p'}(T_B) $. By Theorem \ref{mianth1}, we know that $T_\mu$ is bounded on $f\in A_\alpha^p(T_B)$.
	
	For any $f\in A_\alpha^p(T_B)$, we have
	$$
	\begin{aligned}
	\lVert T_{\mu}f \rVert _{p,\alpha}&=\text{sup}\{|\left< T_{\mu}f,g \right> |\left| \lVert g \rVert _{p',\alpha}=1 \ in\,\,A_{\alpha}^{p'}\left( T_B \right) \right. \}\\
	&=\text{sup}\left\{ \left| \int\limits_{T_B}{f}\overline{g}d\mu \right|\left| \lVert g \rVert _{p',\alpha}=1 \ in\,\,A_{\alpha}^{p'}\left( T_B \right) \right. \right\}\\
	&\le \lVert f \rVert _{L^p\left( \mu \right)}\text{sup}\left\{ \lVert g \rVert _{L^{p'}\left( \mu \right)}\left| \lVert g \rVert _{p',\alpha}=1 \ in\,\,A_{\alpha}^{p'}\left( T_B \right) \right. \right\} 
	,
	\end{aligned}
	$$
	where the rationality of the second equation uses the Fubini's theorem. Since $\mu$ is a vanishing Carlson measure on $A_{\alpha}^{p'}(T_B) $, the second term of the last inequality is finite. 
	
	Now if $f_j \rightarrow 0$ weakly in $A_\alpha^p(T_B)$, then the compactness of the inclusion mapping implies that $\|f\|_{L^p(\mu)} \rightarrow 0$. It follows that $\lVert T_{\mu}f_j \rVert _{p,\alpha} \rightarrow 0$ and hence $T_\mu$ is compact.  
	
	This completes the proof of Theorem \ref{mianth2}.
\end{proof}

\bibliography{reference}
\bibliographystyle{plain}{}
\end{document}